\newfont{\bcb}{msbm10}
\newfont{\matb}{cmbx10}
\newfont{\got}{eufm10}

\documentclass[12pt]{amsart}
\usepackage{amsmath, amsthm, amscd, amsfonts, amssymb, latexsym, graphicx, color}
\usepackage[bookmarksnumbered, colorlinks, plainpages, hypertex]{hyperref}

\usepackage[cp1250]{inputenc}

\usepackage{amsmath,amsthm}
\usepackage{amssymb,latexsym}
\usepackage{enumerate}

\newtheorem{theorem}{Theorem}[section]
\newtheorem{lemma}[theorem]{Lemma}
\newtheorem{proposition}[theorem]{Proposition}
\newtheorem{corollary}[theorem]{Corollary}
\newtheorem{claim}[theorem]{Claim}

\theoremstyle{definition}

\newtheorem{example}[theorem]{Example}

\newtheorem{remark}[theorem]{Remark}
\numberwithin{equation}{section}

\begin{document}

\title[Tame topology in Hensel minimal structures]{Tame topology \\ in Hensel minimal structures}

\author[Krzysztof Jan Nowak]{Krzysztof Jan Nowak}




\date{}

\begin{abstract}
We are concerned with topology of Hensel minimal structures on non-trivially valued fields $K$, whose axiomatic theory was introduced in a recent paper by Cluckers--Halupczok--Rideau. We additionally require that every definable subset in the imaginary sort $RV$, binding together the residue field $Kv$ and value group $vK$, be already definable in the plain valued field language. This condition is satisfied by several classical tame structures on Henselian fields, including Henselian fields with analytic structure, V-minimal fields, and polynomially bounded o-minimal structures with a convex subring. In this article, we establish many results concerning definable functions and sets; among others, existence of the limit for definable functions of one variable, a closedness theorem, several non-Archimedean versions of the \L{}ojasiewicz inequalities, an embedding theorem for regular definable spaces, and the definable ultranormality and ultraparacompactness of definable Hausdorff LC-spaces.
\end{abstract}

\maketitle

\section{Introduction}

We are concerned with geometry and topology of Hensel minimal (more precisely, 1-h-minimal) structures on non-trivially valued fields $K$ of equicharacteristic zero, whose axiomatic theory (in an expansion $\mathcal{L}$ of the language of valued fields) was introduced in the recent papers~\cite{C-H-R,C-H-R-2}. From Section~3 on, we shall additionally assume (unless otherwise stated) the following

\vspace{1ex}

{\bf Condition~(*).}
\begin{em}
Every definable subset in the imaginary sort $RV$, binding together the residue field $Kv$ and value group $vK$, be already definable in the plain valued field language $\mathcal{L}_{hen}$ (see Section~2).
\end{em}

\vspace{1ex}

Note that condition~(*) is satisfied by several classical tame structures on Henselian fields; for instance, by the following four natural examples of Hensel minimal structures:

\vspace{1ex}

1) Henselian valued fields in the plain algebraic language of valued fields ($\omega$-h-minimal).

\vspace{1ex}

2) Henselian valued fields with analytic structure (introduced i the papers~\cite{C-Lip-0,C-Lip}, and $\omega$-h-minimal by \cite{C-H-R}, Theorem~6.2.1).

\vspace{1ex}

3) V-minimal fields ($1$-h-minimal, {\em op.cit.}, Theorem~6.4.2).

\vspace{1ex}

4) polynomially bounded o-minimal structures with a convex subring (1-h-minimal, {\em op.cit.}, Theorem~6.3.4).

\begin{remark}\label{rem-1}
Condition~(*) entails (see Section~2) the property (\dag): the residue field $Kv$ and value group $vK$ are orthogonal, and the sets definable in the sort $vK$ are already definable in the language of ordered groups, even after adding to the language $\mathcal{L}$ an angular component map (if it exists). Note that property~(\dag) is fundamental in our approach. Therefore all the results of our paper hold also for 1-h-minimal structures with this property whenever $\mathcal{L}$ includes an angular component map. An example are almost real closed fields with real analytic structure, which are $\omega$-h-minimal and enjoy property~(\dag) in a language with an angular component map (cf.~\cite{NSV}, Theorems~1.0.1 and~1.0.2).
\end{remark}

In our geometric approach, most essential is which (not how) sets are definable.
The words 0-definable and $A$-definable will mean $\mathcal{L}$-definable and $\mathcal{L}_{A}$-definable; "definable" will refer to definable in $\mathcal{L}$ with arbitrary parameters. Observe that usually the $A$-definable variants of assertions follow immediately from their 0-definable versions because adding constants preserves Hensel minimality (cf.~\cite[Section~2]{C-H-R}).

\vspace{1ex}

In this article, we establish many topological and geometric results concerning definable functions and sets such as, for instance, existence of the limit for definable functions of one variable, a closedness theorem and several non-Archimedean versions of the \L{}ojasiewicz inequalities. In the algebraic case of Henselian fields (also with analytic structure), those results were achieved in our previous papers~\cite{Now-Sel,Now-Sym,Now-Sing,Now-Alant,Now-resonances}.

\vspace{1ex}

We shall also provide an embedding theorem for non-Archimedean regular definable spaces (whose o-minimal and semi-algebraic versions go back to van den Dries~\cite{Dries-2} and Robson~\cite{Rob}),  introduce so called definable LC-spaces, and prove the definable ultranormality and ultraparacompactness of definable Hausdorff LC-spaces and, a fortiori, of definable manifolds.

\vspace{1ex}

This article is organized as follows. In Section~2, we provide some basic model-theoretic terminology and facts (including an algebraic language $\mathcal{L}_{rv}$ for the leading term structure $RV$) concerning valued fields. And next, following the paper~\cite{C-H-R}, some results from Hensel minimality needed in our approach will be recalled.

\vspace{1ex}

From Section~3 on, we shall always assume, unless otherwise stated, that the ground field $K$ of equicharacteristic zero is a model of a 1-h-minimal (complete) $\mathcal{L}$-theory $T$ in an expansion $\mathcal{L}$ of the valued field language $\mathcal{L}_{vf}$ (see Section~2 for details), which satisfies the foregoing condition~(*).

\vspace{1ex}

In Section~3, we prove existence of the limit for definable functions of one variable. (For algebraic versions see~\cite[Section~5]{Now-Sel} and~\cite[Section~5]{Now-Sing}.) We adopt the following notation: $\overline{E}$ and $\partial E := \overline{E} \setminus E$ shall denote the topological closure and frontier of a set $E$, respectively.

\begin{theorem}\label{limit-th1}
Consider a 1-h-minimal field $K$ as indicated above,
a 0-definable function $f:E \to K$ on a subset $E$ of $K$, and the set
$$ W := \partial \,(\mathrm{graph} \, (f)) \cap (\{ 0 \} \times
   \mathbb{P}^{1}(K)) \subset \{ 0 \} \times \mathbb{P}^{1}(K), $$
which is finite by dimension theory in 1-h-minimal structures. Suppose that $0$ is an accumulation point of $E$. Then $W$ is a non-empty set, say
$$ W = \{ w_{1}, \ldots, w_{s} \} \subset \mathbb{P}^{1}(K) = K \cup \{ \infty \}, \ \ s \geq 1, $$
and there exists a partition
$$ E = E_{1} \cup \ldots \cup E_{s} $$
into $s$ disjoint $W$-definable sets (a fortiori, definable over the algebraic closure $\mathrm{acl}\, (\emptyset)$ of\/ $\emptyset$) such that $0$ is their accumulation point and
$$ \lim_{x \rightarrow 0}\, f|E_{i}\, (x) = w_{i}, \ \ i=1,\ldots,s. $$
Moreover, there exists a further $W$-definable partition
$$ E = F_{1} \cup \ldots \cup F_{t}, \ \ t \geq s, $$
finer than the initial one, with the following property. If $F_{j} \subset E_{i}$ and $w_{i} \neq \infty$, then the set
$$ \{ (v(x), v(f(x) - w_{i})): \; x \in F_{j} \setminus \{0 \} \} \subset vK \times (vK \cup \{ \infty \}), \ \ i=1,\ldots,t, $$
is contained either in an affine line with rational slope
$$ \{ (k,l) \in vK \times vK: \ \, q \cdot l = p \cdot k + \gamma \, \} $$
with $p,q \in \mathbb{Z}$, $p,q>0$, $\gamma \in vK$, or in\/
$vK \times \{ \infty \}$.
\end{theorem}

The proof relies on condition (*), domain and range preparation (\cite[Proposition~2.8.6]{C-H-R}) and the fact that every function definable in an ordered abelian group is piecewise linear (\cite{C-H}).

\vspace{1ex}

In Section~4, we prove the following closedness theorem. (For algebraic versions see~\cite[Section~7]{Now-Sel} and~\cite[Section~8]{Now-Sing}.) It has numerous applications in geometry of Henselian fields, allowing us, in particular, to apply resolution of singularities in much the same way as over the locally compact fields. Let us mention that the closedness theorem was inspired by the joint paper~\cite{K-N} with J.~Koll{\'a}r.

\begin{theorem}\label{clo-th}
Given a definable subset $D$ of $K^{n}$, the canonical
projection
$$ \pi: D \times \mathcal{O}_{K}^{m} \longrightarrow D  $$
is definably closed in the $K$-topology, i.e.\ if $A \subset D
\times \mathcal{O}_{K}^{m}$ is a closed 0-definable
subset, so is its image $\pi(A) \subset D$.
\end{theorem}

The proof of this theorem will rely, as in our previous papers~\cite{Now-Sel,Now-Sing}, on existence of the limit for definable functions of one variable, parametrized cell decomposition, condition~(*), and a concept of fiber shrinking (being a weaker version of curve selection). In our earlier approach, fiber shrinking had been achieved via existence of a good semi-line with rational slope for a definable subset in the value group sort. However, J.P.\ Acosta pointed out to us (personal communication) that if the valuation is of infinite rank, such a good semi-line may not exist (see Example~\ref{ex-Acosta}). Therefore, in this paper, we obtain fiber shrinking in a much easier way using instead a natural finite partition.

\vspace{1ex}

\begin{remark}~\label{rem-angular}
The notions of limit, continuity, closedness etc.\ are first order properties. Therefore we can prove the above theorem by passage to elementary extensions and arbitrary parameters. One may thus assume that the Henselian field $K$ under study is $\aleph_{1}$-saturated and, consequently, that an angular component map $\overline{ac}$ (also called a coefficient map, after van den Dries~\cite{Dries}) exists. In this paper, we shall sometimes make use of this fact. Note that adding an angular component map preserves h-minimality. This follows from the resplendency property (\cite[Theorem~4.1.19]{C-H-R}) that $RV$-expansions preserve Hensel minimality, and the fact that adding to the language an angular component map $\overline{ac}$ is equivalent to that of a section $\theta$ of the exact sequence~\ref{exact} with $\overline{ac} = \theta \circ rv$ (cf.~Remark~\ref{rem-orthogonal}).
\end{remark}


\begin{remark}
Also observe that Theorem~\ref{clo-th} may be no longer true after expansion of the language for the leading term structure $RV$, as demonstrated in Example~\ref{ex-2}.
\end{remark}


Section~5 is devoted to several applications including, among others, piecewise and uniform continuity, several non-Archimedean versions of the \L{}ojasiewicz inequalities and H\"{o}lder continuity.
(For algebraic versions see~\cite[Section~9]{Now-Sel} and~\cite[Section~11]{Now-Sing}.)

\vspace{1ex}

Section~6 contains some results on separation of definable sets in an affine space. They are collected in a separate section to expose the methods underlying their proofs, which are similar to those behind the \L{}ojasiewicz inequalities. We prove, in particular, that every definable locally closed subset $X$ of $K^n$ is definably ultranormal (Theorem~\ref{afin-ultranormal}).

\vspace{1ex}

In Section~7, we study non-Archimedean definable spaces and definable LC-spaces, i.e.\ spaces obtained by gluing finitely many definable and definable locally closed subsets of affine spaces $K^n$, respectively. We provide, among others, an embedding theorem for regular definable spaces (Theorem~\ref{embed}), being a non-Archimedean analogue of the one from~\cite[Chapter~10]{Dries-2}. Essential tools applied here are the closedness theorem and separation of definable sets.

\vspace{1ex}

Let us finally comment that soon after o-minimality had become a fundamental concept in real geometry (realizing the postulates of both tame topology and tame model theory), numerous attempts were made to find similar approaches in geometry of valued fields. This has led to various, axiomatically based notions such as C-minimality~\cite{Has-Macph-1,Macph}, P-minimality~\cite{Has-Macph-2}, V-minimality~\cite{H-K}, b-minimality~\cite{C-L}, tame structures~\cite{C-Com-L,C-Fo-L}, and eventually Hensel minimality~\cite{C-H-R}.
Several variants of Hensel minimality are in fact introduced, abbreviated by $l$-h-minimality with $l \in \mathbb{N} \cup \{ \omega \}$. The $l$-h-minimality condition is the stronger, the larger the number $l$ is.

\vspace{1ex}

In the equicharacteristic case, already 1-h-minimality provides, likewise o-minimality does, powerful geometric tools as, for instance, cell decomposition, a good dimension theory or the Jacobian property (an analogue of the o-minimal monotonicity theorem). Actually, the majority of the results from~\cite{C-H-R}, including those applied in our paper, hold for 1-h-minimal theories.

\vspace{1ex}

\section{Valuation- and model-theoretical preliminaries.}
We begin with basic notions from valuation theory. By $(K,v)$ we mean a field $K$ endowed with a valuation $v$. Let
$$ vK, \ \mathcal{O}_{K}, \ \mathcal{M}_{K} \ \ \text{and} \ \ Kv $$
denote the value group, valuation ring, its maximal ideal and residue field, respectively. Let $r: \mathcal{O}_{K} \to Kv$ be the residue map.
In this paper, we shall consider the equicharacteristic zero case, i.e.\ the characteristic of the fields $K$ and $Kv$ are assumed to be zero. For elements $a \in K$, the value is denoted by $va$ and the residue by $av$ or $r(a)$ when $a \in \mathcal{O}_{K}$. Then
$$ \mathcal{O}_{K} = \{ a \in K: \; v(a) \geq 0 \}, \ \ \mathcal{M}_{K} = \{ a \in K: \; v(a) > 0 \}. $$
For a ring $R$, let $R^{\times}$ stand for the multiplicative group of units of $R$. Obviously, $1+\mathcal{M}_{K}$ is a subgroup of the multiplicative group $K^{\times}$. Let
$$ rv: K^{\times} \to G(K) := K^{\times}/(1+\mathcal{M}_{K}) $$
be the canonical group epimorphism. Since $vK \cong K^{\times}/\mathcal{O}_{K}^{\times}$, we get the canonical group epimorphism $\bar{v}: G(K) \to vK$ and the following exact sequence
\begin{equation}\label{exact}
1 \to (Kv)^{\times} \to G(K) \to vK \to 0.
\end{equation}
We put $v(0) = \infty$ and $\bar{v}(0) = \infty$. For simplicity, we shall write
$$ v(a) = (v(a_{1}),\ldots,v(a_{n})) \ \ \text{or} \ \ rv(a) = (rv(a_{1}),\ldots,rv(a_{n})) $$
for an $n$-tuple $a = (a_{1},\ldots,a_{n}) \in K^{n}$.

\vspace{1ex}

Notice now that the following 2-sorted plain valued field language $\mathcal{L}_{hen}$ (with imaginary auxiliary sort $RV$) on Henselian fields $(K,v)$ of equicharacteristic zero, goes back to Basarab~\cite{Bas} and yields (even resplendent) quantifier elimination of valued field quantifiers for the theory of Henselian fields.

\vspace{1ex}

{\em Main sort:} a valued field with the language of rings $(K,0,1,+,-,\cdot)$ or the valued field language $\mathcal{L}_{vf}$ with signature $(K,0,1,+,-,\cdot,\mathcal{O}_{K})$.

\vspace{1ex}

{\em Auxiliary sort:}  $RV(K) := G(K) \cup \{ 0 \}$ with the language specified as follows: (multiplicative) language of groups $(1, \cdot)$ and one unary predicate $\mathcal{P}$ such that
$\mathcal{P}_{K}(\xi) \ \Leftrightarrow \ \bar{v}(\xi) \geq 0$; here we put $\xi \cdot 0 = 0$ for all $\xi \in RV(K)$. The predicate
$$ \mathcal{R}(\xi) \ \ \Longleftrightarrow \ \  [\xi=0 \ \vee \  (\xi \neq 0 \wedge \mathcal{P}(\xi) \wedge \mathcal{P}(1/\xi))] $$
will be construed as the residue field $Kv = Kv$ with the language of rings $(0,1,+,\cdot)$; obviously,
$\mathcal{R}_{K}(\xi) \ \Leftrightarrow \ \bar{v}(\xi) =0$. The sort $RV$ binds together the residue field and value group.

\vspace{1ex}

{\em One connecting map:}  $rv: K \to RV(K), \ \ rv(0) = 0.$

\vspace{1ex}

The valuation ring can be defined by putting $\mathcal{O}_{K} = rv^{-1}(\mathcal{P}_{K})$.
The residue map $r: \mathcal{O}_{K} \to Kv$ will be identified with the map
$$ r(x) = \left\{ \begin{array}{cl}
                        rv(x) & \mbox{ if } \ x \in \mathcal{O}_{K}^{\times}, \\
                        0 & \mbox{ if } \ x \in \mathcal{M}_{K}.
                        \end{array}
               \right.
$$


\begin{remark}\label{rem-plus}
Addition in the residue field $\mathcal{R}_{K} \cup \{ 0 \}$ is the restriction of the following algebraic operation on $RV(K)$:
$$ rv(x) + rv(y) = \left\{ \begin{array}{cl}
                        rv(x+y) & \mbox{ if } \ v(x+y) = \min \{ v(x), v(y) \}, \\
                        0 & \mbox{ otherwise }
                        \end{array}
               \right.
$$
for all $x,y \in K^{\times}$; clearly, we put $\xi + 0 = \xi$ for every $\xi \in RV(K)$.
\end{remark}


\begin{remark}\label{rem-language}
The standard language for the sort $RV$, whose vocabulary has just been introduced, is of course equivalent to the language of rings $(0, 1, +, \cdot)$ from Remark~\ref{rem-plus}. In particular, $\bar{v}(\xi) > 0 \ \Leftrightarrow \ 1 + \xi = 1$.
This language of rings for $RV$ will be denoted by $\mathcal{L}_{rv}$.
\end{remark}


It is well known that exact sequence~\ref{exact} splits whenever the residue field $Kv$ is $\aleph_{1}$-saturated. Then there is a section $\theta: G(K) \to (Kv)^{\times}$ of the monomorphism $\iota: (Kv)^{\times} \to G(K)$, and the map
$$ (\theta,\bar{v}): G(K) \to (Kv)^{\times} \times vK $$
is an isomorphism. Both the sides will be identified by means of this isomorphism. Generally, the existence of such a section $\theta$ is equivalent to that of an angular component map $\overline{ac} = \theta \circ rv$.

\begin{remark}\label{rem-orthogonal}
It is easy to check that the language $\mathcal{L}_{rv}$ with the section $\theta$ is equivalent to the language which consists of two connecting maps
$$ \theta: RV(K) \to Kv, \ \theta(0) =0, \ \ \text{and} \ \  \bar{v}: RV(K) \to vK \cup \{ \infty \}, \ \bar{v}(0) = \infty, $$
of the language of rings $(0,1,+,-,\cdot)$ on the residue field $Kv$, and of the language of ordered groups $(0,+,-,<)$ on the value group $vK$.
\end{remark}

In view of the above remark, the residue field is orthogonal to the value group, i.e.\ every definable subset $C \subset (Kv)^{p} \times (vK)^{q}$ is a finite union of Cartesian products
\begin{equation}\label{orthogonal}
C = \bigcup_{i=1}^{k} \ X_{i} \times Y_{i}
\end{equation}
for some definable subsets $X_{i} \subset (Kv)^{p}$ and $Y_{i} \subset (vK)^{q}$. Observe that the subset $\Lambda$ of $RV(K)^{2}$ from Example~\ref{ex-2} cannot be defined in the language $\mathcal{L}_{rv}$.


\vspace{1ex}

We shall fix a language $\mathcal{L}$ which is an expansion of the language $\mathcal{L}_{vf}$ of valued fields, possibly with some auxiliary imaginary sorts. Consider a model $K$ of a 1-h-minimal (complete) $\mathcal{L}$-theory $T$. For the reader's convenience, we recall below the following two results of Hensel minimality from the paper~\cite{C-H-R}, which are crucial for our approach:

\vspace{1ex}

1) Domain and range preparation ({\em op.cit.}, Proposition~2.8.6), which relies on the valuative Jacobian property ({\em op.cit.}, Lemma.~2.8.5);

\vspace{1ex}

2) Reparameterized cell decomposition ({\em op.cit.}, Theorem~5.7.3 ff.).





\vspace{2ex}

For convenience, we remind the reader that a ball $B$ is said to be 1-next to an element $c \in K$ (see~\cite[Section~2]{C-H-R}) if
$$ B = \{ x \in K: \ rv(x-c)= \xi \} $$
for some $\xi \in RV(K)$, $\xi \neq 0$. A ball $B$ shall be called 1-next to a finite non-empty set $C \subset K$ if
$$ B = \bigcap_{c \in C} \, B_{c} $$
with $B_{c}$ a ball 1-next to $c \in C$ for each $c \in C$; then $B$ is 1-next to some element $c \in C$.

\begin{proposition}\label{dom-range} (Domain and Range Preparation).
Let $f : K \to K$ be a 0-definable function and let $C_0 \subset K$ be a finite, 0-definable set. Then there exist
finite, 0-definable sets $C,D \subset K$ with $C_0 \subset C$ such that $f(C) \subset D$ and for every ball $B$ 1-next to $C$, the image $f(B)$ is either a singleton in $D$ or a ball 1-next to $D$; moreover, the conclusions (1) and (2) of the Valuative Jacobian Property hold.   \hspace*{\fill} $\Box$
\end{proposition}

\vspace{1ex}

For $m \leq n$, denote by $\pi_{\leq m}$ or $\pi_{< m+1}$ the projection $K^{n} \to K^{m}$ onto the first $m$ coordinates; put $x_{\leq m} = \pi_{\leq m}(x)$. Let $C \subset K^{n}$ be a non-empty 0-definable set, $j_{i} \in \{ 0, 1 \}$ and
$$ c_{i} : \pi_{<i}(C) \to K $$
be 0-definable functions $i=1,\ldots,n$. Then $C$ is called a 0-definable cell with center tuple $c = (c_{i})_{i=1}^{n}$ and of cell-type $j =(j_{i})_{i=1}^{n}$ if it is of the form:
$$ C = \left\{ x \in K^{n}: (rv(x_{i} - c_{i}(x_{<i})))_{i=1}^{n} \in R \right\}, $$
for a (necessarily 0-definable) set
$$ R \subset \prod_{i=1}^{n} \, j_{i} \cdot G(K), $$
where $0 \cdot G(K) = {0} \subset RV(K)$ and $1 \cdot G(K) = G(K) \subset RV(K)$. One can similarly define $A$-definable cells.

\vspace{1ex}

In the absence of the condition that algebraic closure and definable closure coincide in $T = \mathrm{Th}\, (K)$ (i.e.\ the algebraic closure $\mathrm{acl}\,(A)$ equals the definable closure $\mathrm{dcl}\,(A)$ for any Henselian field $K' \equiv K$ and every $A \subset K'$), a concept of parameterized cells must come into play. Let us mention that one can ensure the above condition just via an expansion of the language for the sort $RV$ (cf.~\cite[Proposition~4.3.3]{C-H-R}).

\vspace{1ex}

Consider a 0-definable function $\sigma: C \to RV(K)^{k}$. Then $(C,\sigma)$ is called a 0-definable parameterized (by $\sigma$) cell if each set $\sigma^{-1}(\xi)$,  $\xi \in \sigma(C)$, is a $\xi$-definable cell with some center tuple $c_{\xi}$ depending definably on $\xi$ and of cell-type independent of $\xi$.

\vspace{1ex}

\begin{remark}\label{rem-rv}
If the language $\mathcal{L}$ has an angular component map, then one can take $\sigma$ in the above definition to be residue field valued
(instead of $RV$-valued).
\end{remark}

Below we recall a fundamental result on parametrized cell decomposition from~\cite{C-H-R} (Theorem~5.7.3 along with Addenda 1, 2, and 4).

\begin{theorem}\label{cell} (Parameterized Compatible Cell Decomposition)
For every 0-definable sets
$$ X \subset K^{n} \ \ \text{and} \ \ P \subset X \times RV(K)^{t}, $$
there exists a finite decomposition of $X$ into 0-definable parametrized cells $(C_{k},\sigma_{k})$ with continuous centers such that the fibers of $P$ over each twisted box of each $C_{k}$ are constant or, equivalently, the fiber of $P$ over each $\xi \in RV(K)^{t}$ is a union of some twisted boxes from the cells $C_{k}$.
Moreover, given finitely many 0-definable functions $f_{j}: X \to K$, one can require that the restriction of every function $f_{j}$ to each cell $\sigma_{k}^{-1}(\xi)$ be continuous. \ \ \hspace*{\fill} $\Box$
\end{theorem}



\vspace{1ex}

\section{Existence of the limit}
In this section, we prove Theorem~\ref{limit-th1} on existence of the limit for definable functions of one variable. It is easy to reduce the problem to the case where the function $f$ is bounded. Observe that it suffices to show that the set $W \subset K$ is non-empty. Indeed, put
$$ E_{i} := f^{-1}(\{ x \in E: \ v(x-w_{i}) > \rho \}), \ \ i=1,\ldots,s, $$
where
$$ \rho := \max \, \{ v(w_{i}-w_{j}): \ i,j =1,\ldots,s, \ i \neq j \}, $$
and
$$ E_{0} := E \setminus (E_{1} \cup \ldots \cup E_{s}). $$
Then
$$ W_{0} := \partial \,(\mathrm{graph} \, (f|E_{0}) \cap (\{ 0 \} \times K) = \emptyset . $$
Hence and by the assertion applied to the restriction $f|E_{0}$, we see that $0$ is not an accumulation point of the set $E_{0}$. Then
$$ E_{1} \cup E_{0}, E_{2}, \ldots, E_{s} $$
is a partition we are looking for.

\vspace{1ex}

We are going to prove that the set $W$ is non-empty.
By Remarks~\ref{rem-angular} and~\ref{rem-orthogonal} ff., we can pass to elementary extensions and assume that the field $K$ has a coefficient map $\overline{ac}$, exact sequence~\ref{exact} splits, and the residue field and value group are orthogonal. Then we have the isomorphism
$$ (\theta,\bar{v}): G(K) \to (Kv)^{\times} \times vK, $$
and thus we can identify $G(K)$ with $(Kv)^{\times} \times vK$. This isomorphism is of significance because topological properties of the valued field $K$ are described in terms of the value group $vK$.

\vspace{1ex}

We may of course assume that $0 \not \in E$, $E$ is the union of some balls 1-next to a point from a 0-definable finite set $C_{0} \subset K$ with $0 \in C_{0}$ (by parametrized cell decomposition), and that the function $f$ is bounded, with image contained in the open unit ball centered at the origin, $f(E) \subset \mathcal{O}_{K}$. Extend the function $f$ by putting
$$ f(x) =1 \ \ \text{for all} \ \ x \in K \setminus E. $$
By Proposition~\ref{dom-range}, there exist finite 0-definable subsets $C \subset K$ with $C_{0} \subset C$  and $D \subset K$ such that for every ball $B$ 1-next to $C$, the image $f(B)$ is either a singleton in $D$ or a ball 1-next to $D$. Set
$$ \gamma_{0} := \max \, \{ v(c): \ c \in C, \ c \neq 0 \}. $$
Without loss of generality, we can assume that $E$ is contained in the open ball of radius $\gamma_{0}$ centered at the origin, and consequently that $E$ is the union of some balls 1-next to $0$ and to $C$ too.

Therefore we can partition the domain $E$ into a finite number of $\mathrm{acl}\, (\emptyset)$-definable pieces $E_{1},\ldots,E_{s}$ with accumulation point $0$ and points $w_{1},\ldots,w_{s} \in D$ such that if a ball $B \subset E_{i}$ is 1-next to $0$, then the image $f(B)$ is either $\{ d_{i} \}$ or a ball 1-next to $d_{i}$. It suffices to show that
$$ \lim_{x \rightarrow 0}\, f|E_{i}\, (x) = w_{i}, \ \ i=1,\ldots,s. $$
So fix an $i \in \{ 1,\ldots,s \}$; the problem easily reduces to the case $d_{i}=0$. The case $f(B) = \{ d_{i} \}$ is clear. So consider the other one. Obviously, the balls 1-next to $0$ are of the form $\{ rv(x) =\xi \}$ for a unique $\xi \in G(K)$.
Consider the $\mathrm{acl}\, (\emptyset)$-definable set $X \subset G(K)^{2}$ defined by the formula
$$ \left\{ (\xi,\eta) \in G(K)^{2}: \ \{ rv(x) =\xi \} \subset E_{i}, \ f(\{ rv(x) =\xi \}) =  \{ rv(y) =\eta \} \right\} . $$
By the orthogonality property, $X$ is defined by a finite disjunction of conjunctions of the form:
\begin{equation}\label{disjunction}
  \phi(\theta(\xi),\theta(\eta)) \ \wedge \ \psi(\bar{v}(\xi),\bar{v}(\eta)).
\end{equation}
We can assume, without loss of generality, that $X$ is defined by one from those conjunctions and is of the form:
$$ \theta(\eta) = \alpha (\theta(\xi)) \ \wedge \ \bar{v}(\eta) = \beta (\bar{v}(\xi)) $$
for some $\mathrm{acl}\, (\emptyset)$-definable functions $\alpha$ and $\beta$ with domains $\Theta \subset Kv$ and $\Delta \subset vK$, respectively; obviously, the domain $\Delta$ is with accumulation point $\infty$.

\vspace{1ex}

Now we apply the following theorem from~\cite[Corollary~1.10]{C-H} to the effect that functions definable in ordered abelian groups are piecewise linear.

\vspace{1ex}

\begin{proposition}~\label{prop-linear}
Consider an ordered abelian group $G$ with the language of ordered groups $\mathcal{L}_{oag} = ( 0,+, < )$.
Let
$$ \beta :G^{n} \to G $$
be an $A$-definable function for a subset $A \subset G$. Then there exists a partition of $G^{n}$ into finitely many $A$-definable subsets
such that, on each subset $S$ of them, $\beta$ is linear; more precisely, there exist
$$ r_{1},\ldots, r_{n}, s \in \mathbb{Z}, \ \ s \neq 0, $$
and an element $\gamma$ from the definable closure of $A$ such that
$$ \beta(a_{1},\ldots,a_{n}) = \frac{1}{s} \cdot (r_{1} a_{1} + \ldots + a_{n} r_{n} + \gamma) $$
for all $(a_{1},\ldots,a_{n}) \in S$.   \hspace*{\fill} $\Box$
\end{proposition}

\vspace{1ex}

Hence, for $\bar{v}(\xi) \in \Delta$, we obtain the equivalence
$$ \bar{v}(\eta) = \beta (\bar{v}(\xi))  \ \  \Longleftrightarrow \ \
   \bar{v}(\eta) =  \frac{1}{s} \cdot (r \cdot \bar{v}(\xi) + \gamma). $$
Then the set
$$ F := \{ a \in K: \ \overline{ac}(a) \in \Theta, \ v(a) \in \Delta \} $$
is an $\mathrm{acl}\, (\emptyset)$-definable subset of $E$ with accumulation point $0$.
We thus encounter three cases:

\vspace{1ex}

{\em Case 1.} If $r/s > 0$, then
$\lim_{x \rightarrow 0} \, f|F (x) = 0$, and we are done.

\vspace{1ex}

{\em Case 2.} If $r/s < 0$, then
$\lim_{x \rightarrow 0} \, f|F (x) = \infty$, which is impossible since the function $f$ is assumed to be bounded.

\vspace{1ex}

{\em Case 3.} Were $r/s = 0$, then the function $\beta \equiv \gamma/s$ would be constant. Then for any $\theta \in \Theta$ and $\delta \in \Delta$, we would get
$$ f( \{ x: \ \overline{ac}(x) = \theta, \ v(x) = \delta \}) = \{ y: \ \overline{ac}(x) = \alpha(\theta), \ v(y) = \gamma/s \}. $$
Fix a $\lambda \in \alpha(\Theta)$.
Then, for any every point $b \in K$ from the ball
$$ \{ y \in K: \ \overline{ac}(y) =\lambda, \ v(y) = \gamma/s \}, $$
the set $(f|F)^{-1}(b)$ would be infinite, which is impossible. This contradiction shows that Case 3 cannot happen either, and the first conclusion of the theorem follows.

\vspace{1ex}

The second conclusion about a finer $W$-definable partition can be easily obtained by further  finer partitioning of the set $E$ with respect to not only the finite number of disjunctions~\ref{disjunction}, but also application of Proposition~\ref{prop-linear} on piecewise linearity of definable functions. This completes the proof.   \hspace*{\fill} $\Box$

\vspace{2ex}





\section{Proof of the closedness theorem}

As in Section~3, we can pass to elementary extensions and assume that the field $K$ has a coefficient map $\overline{ac}$, exact sequence~\ref{exact} splits, and the residue field and value group are orthogonal.
We begin by stating a lemma on some definable sets in the value group sort.

\begin{lemma}\label{group-sort}
Consider a descending definable family $\Lambda_{\rho} \subset vK$, $\rho \in vK$, of non-empty sets. Suppose that this family is bounded from below and above, say by $-\alpha$ and $\alpha$ for some $\alpha \in vK$, $\alpha >0$. Then it is ultimately constant.
\end{lemma}

\begin{proof}
This lemma requires more care for groups of infinite rank.
It follows directly from relative quantifier elimination for ordered abelian groups from~\cite{C-H} which asserts that every definable subset $\Lambda \subset (vK)^{2}$ is given by a formula in family union form
\begin{equation}\label{family-form}
\bigvee_{i=1}^{k} \, \exists \, \bar{\theta} \ [\xi_{i}(\bar{\theta}) \wedge \psi_{i}(x,y,\bar{\theta})],
\end{equation}
where $\bar{\theta}$ are variables from the auxiliary sorts, $\xi_{i}(\bar{\theta})$ live purely in the auxiliary sorts, and each $\psi_{i}(x,y,\bar{\theta})$ is a conjunction of literals.

Consider the subset
$$ \Lambda := \bigcup_{\rho \in vK} \, \{ \rho \} \times \Lambda_{\rho}. $$
Then $\Lambda = \bigcup_{i=1}^{k} \, \Lambda_{i}$, where $\Lambda_{i}$ is the subset defined by the $i$-th formula of the disjunction~\ref{family-form}, $i=1,\ldots,k$.
We may assume, without loss of generality, that for each $i \in \{ 1,\ldots,k \}$ the set
$$ \{ \rho \in vK: \ \emptyset \neq \Lambda_{i,\rho} := \{ y \in vK: \ (\rho,y) \in \Lambda_{i} \} \} $$
is unbounded from above.

Recall that atomic formulas occurring in the $\psi_{i}(x,y,\bar{\theta})$ are built from predicates for the relations
$$ x_{1} \diamond_{\omega} x_{2} + k_{\omega}, \ \ \diamond \in \{ =,<,\equiv_{m} \}, \ k \in \mathbb{Z}, \ m \in \mathbb{N}, $$
with $\omega$ from the auxiliary sorts, and from the predicates
$$ x_{1} \equiv^{[m']}_{m,\omega} x_{2}, \ \ m,m' \in \mathbb{N}, \ \omega \in \mathcal{S}_{p}, \ p \in \mathbb{P}. $$
In the above predicates one substitutes for $x_{1}, x_{2}$ some terms which are linear functions with integer coefficients in the variables $x,y$.
Therefore, since the subset $\Lambda$ is bounded from below and above, it is not difficult to deduce the following property:

\vspace{1ex}

\begin{em}
Let $\Lambda_{i}'$ be the subset defined by the conjunction of the non-congruence literals occurring in $\psi_{i}(x,y,\bar{\theta})$, thus built from the predicates $=_{\alpha}$ and $<_{\alpha}$. If $\rho > \mathbb{N} \alpha$ and $-\alpha \leq \sigma \leq \alpha$, then whether those non-congruence literals are satisfied by $(\rho, \sigma)$ depends only on $\rho$ and on literals with terms which are linear functions of only the variable $y$. Hence, for each $i=1,\ldots,k$, the non-empty fibers $\Lambda_{i,\rho}'$, $\rho \in vK$, are constant for $\rho > \mathbb{N} \alpha$.
\end{em}

\vspace{1ex}

Clearly, there is a positive integer $M$ such that every literal from the $\psi_{i}(x,y,\bar{\theta})$, $i=1,\ldots,k$, which is built from the congruence predicates $\equiv_{m,\omega}$ and $\equiv^{[m']}_{m,\omega}$, $m,m' \in \mathbb{N}$, is independent of shifts of variables
$$ x \mapsto x + M \gamma, \ \ y \mapsto y + M \delta, $$
for every $\gamma,\delta \in vK$; this means that every such literal is satisfied by $(\alpha,\beta) \in (vK)^{2}$ if and only if it is satisfied by $(\alpha+ M \gamma,\beta + M \delta) \in (vK)^{2}$.

Therefore, taking into account that the family $\Lambda_{\rho}$ is descending and the congruence relations involved in the definition of $\Lambda$ are independent of the shifts of the variable $x$, we can conclude that the initial family $\Lambda_{\rho}$ is ultimately constant. This finishes the proof.
\end{proof}

As in our previous papers~\cite{Now-Sel,Now-Sing}, the proof of the closedness theorem will make use of fiber shrinking, achieved earlier via existence of a good semi-line with rational slope for a definable subset in the value group sort. When the value group is of infinite rank, such a semi-line may not exist, as indicated in the following example by J.P.~Acosta.

\begin{example}\label{ex-Acosta}
Consider the set $I := \mathbb{N} \times \mathbb{N}$ with the lexicographic order and the abelian group $G := \bigoplus_{I} \, \mathbb{Z}$ with order induced by that on $I$ and $\mathbb{Z}$. Denote by $I_{1}$ (respectively $I_{0}$) the sets of elements from $I$ that have (respectively, do not have) predecessors in $I$; and by $P_{1}$ (respectively, $P_{0}$) the sets of elements from $G$ whose dominant term corresponds to an element of $I_{1}$ (respectively, of $I_{0}$). Then for any affine function
$$ f(x) = rx + \tau, \ \ r \in \mathbb{N} \setminus \{ 0 \}, \ \tau \in G, $$
we have
$$ f(P_{1} \cap G_{>\alpha}) \subset P_{1} \ \ \text{and} \ \ f(P_{0} \cap G_{>\alpha}) \subset P_{0} $$
for some $\alpha \in G$, where $G_{>\alpha} := \{x \in G: \ x > \alpha \}$. Therefore there is no semi-line
$$ L = \{ (r_{1}\tau + \gamma_{1}, r_{2}\tau + \gamma_{2}): \, \tau
   \in G, \ \tau \geq 0 \} $$
with $r_{1},r_{2} \in \mathbb{N} \setminus \{ 0 \}$, $\gamma_{1},\gamma_{2} \in G$, and such that $(\infty,\infty)$ is an accumulation point of the intersection $(P_{0} \times P_{1}) \cap L$.
\end{example}

But actually the proof of fiber shrinking from~\cite[Proposition~6.1]{Now-Sel} can be simplified by replacing the argument with a good semi-line $L$  with a finite partition in the value group sort, as described below.

\vspace{1ex}

Consider a 0-definable subset $A \subset K^{n}$ with accumulation point $0 \in K^{n}$. Put
$$ P := \{ (v(x_{1}),\ldots,v(x_{n})) \in (vK)^{n}: \ (x_{1},\ldots,x_{n}) \in A \}, $$
$$ \Theta_{i} := \{ \gamma \in (vK)_{\geq 0}^{n}: \ \forall \, j=1,\ldots,n \ \gamma_{j} \geq \gamma_{i} \}, $$
and
$$ \widetilde{\Theta_{i}} := \{ x \in K^{n}: \ (v(x_{1}),\ldots,v(x_{n})) \in \Theta_{i} \}. $$
Then $(vK)_{\geq 0}^{n} = \bigcup_{i}^{n} \, \Theta_{i}$ and $(\infty,\ldots,\infty)$ is an accumulation point of one of the sets $P \cap \Theta_{i}$, $i=1,\ldots,n$; say of $P \cap \Theta_{1}$. Therefore
$$ \Phi := A \cap \widetilde{\Theta_{1}} $$
is a 0-definable $x_{1}$-fiber shrinking for the set A at $0$.

\vspace{1ex}

Now we can turn to the {\em proof of the closedness theorem}.
As observed before (\cite[Section~7]{Now-Sel}), fiber shrinking makes it possible to reduce the proof to the case $m=n=1$, which will now be considered.

\vspace{1ex}

We must thus show that if $A$ is a 0-definable subset of $D \times \mathcal{O}$ and a point $b=0 \in K$ lies in the
closure of the projection $B := \pi_{1}(A)$, then there is a point $a \in A$ such that $b = \pi_{1}(a) = 0$.
We still need the following

\begin{lemma}\label{family}
Consider a 0-definable family $X_{\xi}$, $\xi \in (Kv)^{k}$, of subsets of $K^{n}$ and a point $a \in K^{n}$. Then $a$ lies in the closure of the union $\bigcup_{\xi} \, X_{\xi}$ iff $a$ lies in the closure of $X_{\xi_{0}}$ for some $\xi_{0}$.
\end{lemma}

\begin{proof}
Apply the orthogonality property (for the residue field $Kv$ and value group $vK$) to the set
$$ \bigcup_{\xi \in (Kv)^{k}} \, \{ \xi \} \times v(X_{\xi} - a). $$
\end{proof}

Hence and by decomposition into cells with residue field valued parametrization, we are reduced to the case where $A$ is the closure of a $\xi$-definable cell $C_{\xi}$ of a type $(1,j_{2})$ for some $\xi \in (Kv)^{k}$:
$$ C_{\xi} = \left\{ x \in K^{2}: \ (rv(x_{1} -c_{1}), rv(x_{2} - c_{2}(x_{1}))) \in R \right\} $$
with a continuous center $c$ and a $\xi$-definable set
$$ R \subset \prod_{i=1}^{2} \, j_{i} \cdot G(K), \ \ j_{i} \in \{ 0,1 \}. $$
The case $j_{2}=0$ is obvious by virtue of Theorem~\ref{limit-th1}.

\vspace{1ex}

Now consider the case $j_{2}=1$. If $c_{1} \neq 0$, then $0 \in B =  \pi_{<2}(C_{\xi})$ and the theorem follows. Suppose $c_{1}=0$. By Theorem~\ref{limit-th1}, we can assume that the center $c_{2}(x_{1})$ extends to a continuous function at $c_{1}=0$, denoted by the same letter for simplicity. Then we can assume the center $c_{2}(x_{1})$ vanishes.

\vspace{1ex}

If the point $(0,0)$ lies in the closure of $A = C_{\xi}$, we are done.

\vspace{1ex}

Otherwise there is an $\epsilon \in vK$, $\epsilon  \geq 0$, such that
$$ [\, (x_{1},x_{2}) \in A \ \wedge \ v(x_{1}) > \epsilon \, ] \ \Longrightarrow \ v(x_{2}) \leq \epsilon. $$
Then every ball in $C_{\xi}$ lying over the points $x_{1} \in B$ with $v(x_{1}) > \epsilon$ is of radius $\leq \epsilon$. 
Then
$$ C_{\xi} = \left\{ x \in K^{2}: \ (rv(x_{1}), rv(x_{2})) \in R \right\} $$
for a subset $R$ of $G(K) \times G(K)$ such that the set
$$ \bar{v}(R) := \{ (\bar{v}(\eta_{1}), \bar{v}(\eta_{2})): \ (\eta_{1},\eta_{2}) \in R \} \subset vK \times (vK)_{\geq 0} $$
is bounded.
Again, by the orthogonality property, $R$ is a finite union of Cartesian products
$$ R = \bigcup_{i=1}^{k} \ X_{i} \times Y_{i} $$
for some non-empty definable subsets
$$ X_{i} \subset Kv \times Kv \ \ \text{and} \ \ Y_{i} \subset vK \times (vK)_{\geq 0}. $$
Then the definable families $\Lambda_{i,\rho}$, $\rho \in vK$, of the projections onto the second factor of the sets
$$ Y_{i} \cap ((\rho,\infty) \times vK), \ \ i=1,\ldots,k, $$
are descending.

\vspace{1ex}

Since $0$ is an accumulation point of the cell $B$, at least one of them, say $\Lambda_{1,\rho}$, $\rho \in vK$, is a family of non-empty sets. By Lemma~\ref{group-sort}, it is ultimately constant, say equal to $\Lambda_{1}$. Let $\Delta_{1}$ be the projection onto the second factor of the set $X_{1}$. Then, for every point $x_{2} \in K$ such that $rv(x_{2}) \in \Delta_{1} \times \Lambda_{1}$, we get $(0,x_{2}) \in A = \overline{C_{\xi}}$. This completes the proof of Theorem~\ref{clo-th}.    \hspace*{\fill} $\Box$

\vspace{2ex}

We still give an example which demonstrates that the closedness theorem may fail after expansion by predicates of the language for the leading term structure $RV$. Notice that such expansions remain Hensel minimal by virtue of the resplendency of Hensel minimality (cf.~\cite[Theorem~4.1.19]{C-H-R}).


\begin{example}\label{ex-2}
Suppose that the exact sequence~\ref{exact} splits and the value group $vK = \mathbb{Z}$. We thus have a (non-canonical) isomorphism
$$ G(K) \simeq (Kv)^{\times} \times vK $$
(cf.~Remarks~\ref{rem-language} and~\ref{rem-orthogonal}). Consider the language $\mathcal{L}_{rv}^{\prime}$ which is an expansion of the language $\mathcal{L}_{rv}$ with a section $\theta$ by a predicate to identify $vK$ as a subset of $RV$. Next augment the language $\mathcal{L}_{rv}^{\prime}$ by a predicate to name the set
$$ \Lambda := \{ (\lambda,\xi) \in RV(K)^{2}: \ \theta(\lambda)=1, \, \bar{v}(\lambda) \in \mathbb{N}, \, \theta(\xi) \in \mathbb{N}, \, \bar{v}(\xi)=0 \}. $$
Then the set
$$ A := \left\{ (x,y) \in K^{2}: \ rv(x,y) \in \Lambda \right\} $$
is a closed subset of $K^{2}$ definable in the augmented language, but its projection
$$ \pi(A) = \{ x \in K: \ rv(x) = (1,k), \ k \in \mathbb{N} \} $$
is not a closed subset of $K$, having $0 \in K$ as an accumulation point. Observe finally that the set $\Lambda$ is not definable (even with parameters) in the language $\mathcal{L}_{rv}^{\prime}$.
   \hspace*{\fill} $\Box$
\end{example}

\vspace{1ex}

\section{Applications of the closedness theorem}

We begin by proving piecewise continuity.

\begin{proposition}\label{piece}
Let $A \subset K^{n}$ and $f: A \to \mathbb{P}^{1}(K)$ be an
$0$-definable function. Then $f$ is piecewise
continuous, i.e.\ there is a finite partition of $A$ into
$0$-definable locally closed subsets
$A_{1},\ldots,A_{s}$ of $K^{n}$ such that the restriction of $f$
to each $A_{i}$ is continuous.
\end{proposition}

\begin{proof}
Consider the graph
$$ E := \{ (x,f(x)): x \in A \} \subset K^{n} \times \mathbb{P}^{1}(K). $$
We proceed with induction with respect to the dimension
$$ d = \dim A = \dim \, E. $$

\vspace{1ex}

Observe first that every $0$-definable subset $X$ of
$K^{n}$ is a finite disjoint union of locally closed
$0$-definable subsets of $K^{n}$. This can be easily
proven by induction on the dimension of $X$. Therefore we can assume that the graph $E$
is a locally closed subset of $K^{n} \times \mathbb{P}^{1}(K)$ of
dimension $d$ and that the conclusion of the theorem holds for
functions with source and graph of dimension $< d$.

\vspace{1ex}

Let $F := \overline{E}$ be the closure of $E$ in $K^{n} \times \mathbb{P}^{1}(K)$
and $\partial E := F \setminus E$ be its frontier.
Since $E$ is a locally closed subset of $K^{n} \times \mathbb{P}^{1}(K)$, the frontier $\partial E$ is a closed
subset of $K^{n} \times \mathbb{P}^{1}(K)$. Let
$$ \pi: K^{n} \times \mathbb{P}^{1}(K) \longrightarrow K^{n} $$
be the canonical projection. Then, by virtue of the closedness
theorem, the images $\pi(F)$ and $\pi(\partial E)$ are closed
subsets of $K^{n}$. Further,
$$ \dim \, F = \dim \, \pi(F) = d $$
and
$$ \dim \, \pi(\partial E) \leq \dim \, \partial E < d. $$
Putting
$$ B := \pi(F) \setminus \pi(\partial E) \subset \pi(E) = A, $$
we thus get
$$ \dim \, B = d \ \ \text{and} \ \ \dim \, (A \setminus B) < d.
$$
Clearly, the set
$$ E_{0} := E \cap (B \times \mathbb{P}^{1}(K)) = F \cap (B \times
   \mathbb{P}^{1}(K)) $$
is a closed subset of $B \times \mathbb{P}^{1}(K)$ and is the
graph of the restriction
$$ f_{0}: B \longrightarrow \mathbb{P}^{1}(K) $$
of $f$ to $B$. Again, it follows immediately from the closedness
theorem that the restriction
$$ \pi_{0} : E_{0} \longrightarrow B $$
of the projection $\pi$ to $E_{0}$ is a definably closed map.
Therefore $f_{0}$ is a continuous function. But, by the induction
hypothesis, the restriction of $f$ to $A \setminus B$ satisfies
the conclusion of the theorem, whence so does the function $f$.
This completes the proof.
\end{proof}

\begin{remark}
The above proposition can be also achieved by means of~\cite[Theorem~5.1.1]{C-H-R} and dimension theory developed there.
\end{remark}

Yet another direct consequence of the closedness theorem is the following

\begin{proposition}\label{clo-bound}
Let $f:E \to K^m$ be a continuous definable map on a closed bounded subset $E$ of $K^n$. Then the image $f(E)$ is a closed bounded subset of $K^m$ too.
\end{proposition}

\begin{proof}
Consider $f$ as a continuous map into the projective space $\mathbb{P}^m(K)$ and apply the closedness theorem to the graph $F$ of the map $f$:
$$ F := \{ (x,y) \in E \times \mathbb{P}^m(K): \ y=f(x) \}. $$
\end{proof}

Algebraic non-Archimedean versions of the \L{}ojasiewicz
inequalities, established in our papers~\cite{Now-Sel,Now-Sing}, can
be carried over to the general settings considered here with
proofs repeated almost verbatim. Therefore, below they will be only stated without proofs.
These are results~\ref{Loj1}, \ref{Hol}, \ref{UC}, \ref{Loj2} and \ref{Loj3} being Hensel minimal counterparts of the algebraic results 11.2, 11.3, 11.4, 11.5 and 11.6 from our paper~\cite{Now-Sing}, respectively.
The main ingredients of the proofs are the closedness theorem, the orthogonality property and relative
quantifier elimination for ordered abelian groups. They allow us
to reduce the problems under study to that of piecewise linear
geometry. We first state the version, which is closest to the classical one.

\begin{theorem}\label{Loj1}
Let $f,g_{1},\ldots,g_{m}: A \to K$ be continuous
definable functions on a closed bounded subset $A$ of $K^{m}$. If
$$ \{ x \in A: g_{1}(x)= \ldots =g_{m}(x) =0 \} \subset \{ x \in A: f(x)=0 \}, $$
then there exist a positive integer $s$ and a constant $\beta \in vK$ such that
$$ s \cdot v(f(x)) + \beta \geq \min \, \{ v(g_{1}(x)), \ldots ,v(g_{m}(x)) \} $$
for all $x \in A$. Equivalently, in the multiplicative convention, there is a $C \in |K|$ such that
$$ | f(x) |^{s} \leq C \cdot | (g_{1}(x), \ldots , g_{m}(x)) |  $$
for all $x \in A$; here
$$ | (g_{1}(x), \ldots , g_{m}(x)) | := \max \, \{ | g_{1}(x) |, \ldots , |g_{m}(x) | \}. $$   \hspace*{\fill} $\Box$
\end{theorem}

A direct consequence of Theorem~\ref{Loj1} is the following result
on H\"{o}lder continuity of definable functions.

\begin{proposition}\label{Hol}
Let $f: A \to K$ be a continuous definable function
on a closed bounded subset $A \subset K^{n}$. Then $f$ is
H\"{o}lder continuous with a positive integer $s$ and a constant
$\beta \in vK$, i.e.\
$$ s \cdot v(f(x) - f(z)) + \beta \geq  v(x-z) $$
for all $x,z \in A$. Equivalently, there is a $C \in |K|$ such
that
$$ | f(x) - f(z) |^{s} \leq C \cdot | x-z | $$
for all $x,z \in A$.   \hspace*{\fill} $\Box$
\end{proposition}

We immediately obtain

\begin{corollary}\label{UC}
Every continuous definable function $f: A \to K$ on
a closed bounded subset $A \subset K^{n}$ is uniformly continuous.   \hspace*{\fill} $\Box$
\end{corollary}

Now we formulate another, more general version of the
\L{}ojasiewicz inequality for continuous definable functions of a
locally closed subset of $K^{n}$.

\begin{theorem}\label{Loj2}
Let $f,g: A \to K$ be two continuous $0$-definable
functions on a locally closed subset $A$ of $K^{n}$. If
$$ \{ x \in A: g(x)=0 \} \subset \{ x \in A: f(x)=0 \}, $$
then there exist a positive integer $s$ and a continuous
$0$-definable function $h$ on $A$ such that $f^{s}(x) =
h(x) \cdot g(x)$ for all $x \in A$.    \hspace*{\fill} $\Box$
\end{theorem}

Now put
$$ \mathcal{D}(f) := \{ x \in A: f(x)
   \neq 0 \} \ \ \text{and} \ \ \mathcal{Z}\, (f) := \{ x \in A: f(x) = 0 \}. $$
The following theorem may be also regarded as a kind of the
\L{}ojasiewicz inequality, which is, of course, a strengthening of
Theorem~\ref{Loj2}.

\begin{theorem}\label{Loj3}
Let $f: A \to K$ be a continuous $0$-definable function
on a locally closed subset $A$ of $K^{n}$ and $g: \mathcal{D}(f)
\to K$ a continuous $0$-definable function. Then $f^{s}
\cdot g$ extends, for $s \gg 0$, by zero through the set
$\mathcal{Z}\, (f)$ to a (unique) continuous
$0$-definable function on $A$.    \hspace*{\fill} $\Box$
\end{theorem}

Finally notice that the \L{}ojasiewicz inequalities play an important role in geometry of definable sets. Let us mention, for instance, that Theorem~\ref{Loj3} is an essential ingredient of the proof of the Nullstellensatz for regulous (i.e.\ continuous and rational) functions on $K^{n}$ (cf.~\cite[Section~12]{Now-Sel} and~\cite[Section~12]{Now-Sing}).

\vspace{1ex}

\section{Separation of definable sets}

We now prove some results concerning separation of definable sets, which will be applied in the next sections. A definable set $X$ is called  definably ultranormal (in other words, definably zero-dimensional with respect to the large inductive dimension) if, for every two disjoint definable closed subsets $A$ and $B$ of $X$, there exists a definable clopen subset $C$ of $X$ such that $A \subset C$ and $B \subset X \setminus C$.

\begin{theorem}\label{afin-ultranormal}
Every definable locally closed subset $X$ of $K^n$ is definably ultranormal.
\end{theorem}

\begin{proof}
Let $A$ and $B$ be two disjoint closed definable subsets of $X$. For any $\beta \in vK$, $\beta > 0$, put
$$ X_\beta := \{ x \in X: \ v(x) > -\beta, \ \forall \, y \in \partial X \ v(x-y) < \beta \}, $$
$$ A_\beta := A \cap X_\beta, \ \  B_\beta := B \cap X_\beta, $$
and
$$ \Lambda := \{ (\beta, v(x-y)) \in (vK)^2: \ x \in A_\beta, \ y \in B_\beta \}. $$
It is easy to check that $X_\beta$, $A_\beta$ and $B_\beta$ are closed bounded subsets of $K^n$, and that
$$ \bigcup_{\beta >0} \, X_{\beta} = X. $$
It follows from Proposition~\ref{clo-bound} that every set
$$ A_\beta - B_\beta := \{ a-b \in K^n: \ a \in A_\beta, \ b \in B_\beta \} $$
is a closed subset of $K^n$.
Therefore, since
$0 \not \in A_\beta - B_\beta$,
the fibers
$$ \{ \gamma \in vK: (\beta,\gamma) \in \Lambda \} $$
of $\Lambda$ over $\beta$ are bounded, i.e.\ smaller than some $\alpha(\beta) \in vK$.
Now observe that, similarly as in the proofs of the \L{}ojasiewicz inequalities (see \cite[Section~9]{Now-Sel} and~\cite[Section~11]{Now-Sing}), the set $\Lambda$ can be ultimately separated from infinity by a semi-line, which means that
$$ \Lambda \cap \{ (\beta,\gamma) : \beta > \beta_{0} \} \subset \{ (\beta, \gamma) \in (vK)^{2}: \ \gamma < s \cdot \beta \} $$
for a non-negative integer $s$ and some $\beta_{0} \in vK$. Then the set
$$ U := \bigcup_{\beta > \beta_{0}} \  \left( A_\beta + \{ x \in K^n: \ v(x) > s \beta \} \right) $$
is a clopen subset of $K^n$ such that $A \subset U$ and $B \subset K^n \setminus U$, concluding the proof.
\end{proof}

We immediately obtain

\begin{corollary}
The affine space $K^{n}$ is definably ultranormal.
\end{corollary}

\begin{remark}\label{def-separate}
By condition (*) imposed on the auxiliary sort $RV$, and thus on the value group $vK$ too, it is clear that the value $\beta_{0}$ in the above proof can be taken 0-definable whenever the closed subsets $A$ and $B$ are 0-definable. Therefore the subset $U$ is then 0-definable as well.
\end{remark}

In the next section, we shall still need the following generalization of separation of closed definable subsets, whose proof is a straightforward adaptation of the one of the above theorem.

\begin{proposition}\label{separation}
Consider two closed 0-definable subsets $A$ and $B$ of the affine space $K^{n}$. Then there is a closed 0-definable subset $\Omega$ of $K^{n}$ such that
$A \subset \Omega$, $B \setminus A \subset K^{n} \setminus \Omega$, and $\Omega \setminus (A \cap B)$ is a clopen subset of $K^{n} \setminus (A \cap B)$.
     \hspace*{\fill} $\Box$
\end{proposition}

\begin{remark}\label{symmetry}
The situation described in the above proposition is symmetric. Indeed, it is easy to check that the set
$$ U := (K^{n} \setminus \Omega) \cup (A \cap B) $$
is a required clopen subset for the reverse pair of subsets $B$ and $A$.
\end{remark}

\vspace{1ex}

\section{Definable spaces and embedding theorem, definable ultranormality and ultraparacompactness}

In this section we deal with definable spaces $X$, which are defined by gluing finitely many affine definable sets (i.e.\ definable subsets of affine spaces $K^n$). Their theory, developed by van den Dries (cf.~\cite{Dries-2}) in the case of o-minimal structures, carries over to the non-Archimedean settings.
Most natural examples of such spaces are projective spaces, their products and definable subspaces. Obviously, the affine spaces $K^n$ are zero-dimensional with respect to the small inductive dimension; and so are their subspaces since regularity is a hereditary property. Therefore every regular definable space $X$ is zero-dimensional too.

\vspace{1ex}

We shall establish a non-Archimedean version of the embedding theorem for definable spaces in o-minimal structures by L.~van Den Dries~\cite[Chapter~10]{Dries-2}, which goes back to R.~Robson~\cite{Rob} in the semialgebraic case.

\vspace{1ex}

Further, we introduce the concept of {\em definable LC-spaces}, i.e.\ those definable spaces which are defined by gluing finitely many definable, locally closed subsets of affine spaces $K^n$. Such spaces include, in particular, definable manifolds obtained by gluing definable open subsets of an affine space $K^n$.

\vspace{1ex}

We shall prove (Theorem~\ref{ultranormal}) that every definable Hausdorff LC-space $X$ is even definably ultranormal (for the definition see the beginning of Section~6).
This refers in particular to definable manifolds.

\vspace{1ex}

We first give an example of a definable Hausdorf space which is not regular.

\begin{example}
Construct a definable space $X$ by gluing the following two definable subsets of $K^2$ by means of the identity charts:
$$ U_{1} := \left( K^2 \setminus (K \times \{ 0 \}) \right) \cup \{ (0,0) \}, \ \ \ U_{2} := \left( K^2 \setminus (\{ 0 \} \times K) \right). $$
It is not difficult to check that $X$ is a Hausdorff space. Then
$$ A := (K \times \{ 0 \}) \setminus \{ (0,0) \} \subset U_2 $$
is a closed definable subset of $X$, since $A \cap U_1 = \emptyset$ and
$$ A \cap U_2 =  (K \times \{ 0 \}) \cap U_2 $$
is a closed subset of $U_2$. But any neighborhood of $A$ in $U_2$ has $(0,0)$ as an accumulation point. Therefore $A$ and $(0,0)$ cannot be separated by open neighborhhods, and thus $X$ is not a regular definable space.    \hspace*{\fill} $\Box$
\end{example}

\begin{theorem}\label{embed}
Every regular definable space $X$ is affine, i.e.\ $X$ can be embedded into an affine space $K^N$.
\end{theorem}

\begin{proof}
Consider a definable atlas $(\phi_{i}: U_{i} \to V_{i})_{i=1}^{k}$ of $X$ with definable sets $V_{i} \subset K^{n_{i}}$. It suffices to construct a refinement of the covering $(U_{i})_{i=1}^{k}$ which consists of clopen definable subsets of $X$. We may assume that $k=2$ (by an induction argument) and each $V_{i}$ is bounded.

\vspace{1ex}

The idea is to apply the closedness theorem along with Proposition~\ref{separation} about separation, applied to some closed definable subsets of $K^{n_{1}} \times K^{n_{2}}$.
To this end, we adopt the following notation from~\cite[Chapter~10, \S~1]{Dries-2}:
$$ V_{12} := \phi_{1}(U_{1} \cap U_{2}), \ \ V_{21} := \phi_{2}(U_{2} \cap
   U_{1}), $$
$$ B_{1} := V_{1} \cap \partial V_{12} = \phi_{1}(\partial U_{2}), \ \ B_{2} := V_{2} \cap \partial V_{21} = \phi_{2}(\partial U_{1}), $$
$$ B_{1}^{\prime} := \{ x \in K^{n_{1}}: \ \exists \, y \in B_{2} \, \forall \, \epsilon >0 \, \exists \, z \in U_{1} \cap U_{2} \, : $$
$$ [ \, |x - \phi_{1}(z)| < \epsilon, \ |y - \phi_{2}(z)| < \epsilon \, ] \}, $$
$$ B_{2}^{\prime} := \{ y \in K^{n_{2}}: \ \exists \, x \in B_{1} \, \forall \, \epsilon >0 \, \exists \, z \in U_{1} \cap U_{2} \, : $$
$$ [ \, |x - \phi_{1}(z)| < \epsilon, \ |y - \phi_{2}(z)| < \epsilon \, ] \}; $$
here $\partial \, U_{i} := \overline{U_{i}} \setminus U_{i}$ denotes the frontier of $U_{i}$ and $\overline{U_{i}}$ the topological closure of $U_{i}$ in $X$, $i=1,2$. Obviously, $B_{1}$ and $B_{2}$ are closed subsets of $V_{1}$ and $V_{2}$, respectively. The frontiers $\partial \, U_{1}$ and $\partial \, U_{2}$ are disjoint. Indeed,
$$ \partial U_{1} \cap \partial U_{2} = (\overline{U_{1}} \cap (X \setminus U_{1})) \cap (\overline{U_{2}} \cap (X \setminus U_{2})) \subset $$
$$ (X \setminus U_{1}) \cap (X \setminus U_{2}) = X \setminus (U_{1} \cup U_{2}) = X \setminus X = \emptyset. $$

\vspace{1ex}

Further put $\psi_{i} := \phi_{i}^{-1}: V_{i} \to U_{i}$, $i=1,2$,
$$ W:= \{ (x,y) \in V_{12} \times V_{21}: \ \psi_{1}(x) = \psi_{2}(y) \}
   \subset K^{n_{1}} \times K^{n_{2}}, $$
$$ D_{1} := \overline{W} \cap (B_{1} \times K^{n_{2}}) \ \ \text{and} \ \
   D_{2} := \overline{W} \cap (K^{n_{1}} \times B_{2}). $$
Then
$$ B_{1}^{\prime} = p_{1}(D_{2}) \ \ \text{and} \ \ B_{2}^{\prime} =
   p_{2}(D_{1}), $$
where $p_{1}: K^{n_{1}} \times K^{n_{2}} \to K^{n_{1}}$ and $p_{2}: K^{n_{1}} \times K^{n_{2}} \to K^{n_{2}}$ are the canonical projections. It follows from the closedness theorem (Theorem~\ref{clo-th}) that
$p_{1}(D_{1}) = B_{1}$ and $p_{2}(D_{2}) = B_{2}$.

\vspace{1ex}

Under the assumption of regularity, it has been already shown in~\cite[Chapter~10, Claim~1]{Dries-2} that
$$ V_{1} \cap \overline{B_{1}^{\prime}} = \emptyset = V_{2} \cap
   \overline{B_{2}^{\prime}}. $$
Therefore it is not difficult to deduce that
$$ \overline{D_{1}} \cap D_{2} = \emptyset = D_{1} \cap \overline{D_{2}}, $$
and hence
\begin{equation}\label{inclusion}
D_{2} \subset \overline{D_{2}} \setminus \overline{D_{1}} \ \ \text{and} \ \ D_{1} \subset \overline{D_{1}} \setminus \overline{D_{2}}.
\end{equation}

\vspace{1ex}

It follows from Proposition~\ref{separation}, applied to the closed 0-definable subsets $\overline{D_{1}}$ and $\overline{D_{2}}$, that there exists a closed 0-definable subset $\Omega$ of $K^{n_{1}} \times K^{n_{2}}$ such that
$$ \overline{D_{1}} \subset \Omega, \ \ \ \overline{D_{2}} \setminus \overline{D_{1}} \subset (K^{n_{1}} \times K^{n_{2}}) \setminus \Omega, $$
and $\Omega \setminus (\overline{D_{1}} \cap \overline{D_{2}})$ is a clopen subset of $(K^{n_{1}} \times K^{n_{2}}) \setminus (\overline{D_{1}} \cap \overline{D_{2}})$.
Hence and by inclusions~\ref{inclusion}, we get
\begin{equation}\label{emptyset}
\Omega \cap D_{2} = \emptyset.
\end{equation}

\vspace{1ex}

We shall have established the theorem once we prove the following

\begin{claim}
The sets
$$ C_{1} := (U_{1} \setminus U_{2}) \cup (\psi_{1} \circ p_{1})(W \cap \Omega) $$
and
$$ C_{2} := (U_{2} \setminus U_{1}) \cup (\psi_{2} \circ p_{2})(W \setminus \Omega) $$
are disjoint clopen 0-definable neighborhoods in $X$ of the subsets $\partial U_{2}$ and $\partial U_{1}$, respectively.
\end{claim}

Indeed, then $\{ C_{1}, C_{2} \}$ is a finer, 0-definable clopen covering (even partition) of $X$ we are looking for.

\vspace{1ex}

We turn to the above claim. By symmetry (cf.~Remark~\ref{symmetry}), it is enough to prove that $C_{1}$ is a closed subset of $X$. Obviously, $U_{1} \setminus U_{2}$ is a closed subset of $X$. It remains to show that
$$ F_{1} := (\psi_{1} \circ p_{1})(W \cap \Omega) $$
is a closed subset of $X$. So, given an accumulation point $c \in X$ of $F_{1}$, we must show that $c \in F_{1}$.

\vspace{1ex}

To this end, observe that $F_{1}$ is a closed subset of
$$ U_{1} \cap U_{2} = (\psi_{1} \circ p_{1})(W) $$
by virtue of the closedness theorem. We thus encounter two cases: $c \in \partial U_{2}$ or $c \in \partial U_{1}$. The former is clear since $\partial U_{2} \subset U_{1} \setminus U_{2}$.

\vspace{1ex}

Now we prove by reductio ad absurdum that the latter case is impossible. Again, it follows directly from the closedness theorem that $c = (\psi_{2} \circ p_{2})(a)$ for some point
$$ a \in \overline{W \cap \Omega} \subset \overline{W} \cap \Omega. $$
Hence we get
$$ a \in \Omega \cap \overline{W} \cap (K^{n_{1}} \times B_{2}) = \Omega \cap D_{2} = \emptyset; $$
the last equality is just \ref{emptyset}. This contradiction completes both the proof of the above claim and of Theorem~\ref{embed}.
\end{proof}

We now turn to the theory of definable LC-spaces.

\begin{proposition}\label{prop-reg}
Every definable Hausdorff LC-space $X$ is regular.
\end{proposition}

\begin{proof}
Clearly, the details being left to the reader, we shall have established the proposition if we prove the following lemma, wherein the set $W$ will play the role of an auxiliary neighborhood of a point to be separated from a closed definable subset.

\begin{lemma}
Consider a definable chart
$$ (U_1,\phi_1), \ \ \phi_1: U_1 \to V_{1}, $$
where $V_{1}$ is a locally closed subset of $K^{n_1}$. Let $W$ be a definable subset of $U_1$ such that $Z := \phi_1(W)$ is a closed bounded subset of $K^{n_1}$. Then $W$ is a closed subset of $X$.
\end{lemma}

\begin{proof}
Let $a \in X$ be an accumulation point of $W$, and suppose that $a$ lies in a chart
$$ (U_2,\phi_2), \ \ \phi_2: U_2 \to V_{2}, $$
where $V_{2}$ is a locally closed subset of $K^{n_2}$;
obviously, $a$ is an accumulation point of $W \cap U_{2}$. Since $X$ is a Hausdorff space, the (topological) fiber product
$$ P := Z \times_{X} V_{2} = \{ (v_{1},v_{2}) \in Z \times V_{2}: \ \phi_{1}^{-1}(v_{1}) = \phi_{2}^{-1}(v_{2}) \} \subset Z \times V_{2} $$
is a closed definable subset of $Z \times V_{2}$, and thus of $K^{n_{1}} \times V_{2}$ as well. By the closedness theorem, the canonical projection $\pi (P)$ onto the second factor is a closed subset of $V_{2}$. Hence and because $\phi_{2}(a)$ is an accumulation point of $\pi (P)$ by our initial assumption, there is a point $y \in Z$ such that $(y,\phi_{2}(a)) \in P$. Then $a = \phi_{1}^{-1}(y) \in W$, which is the desired result.
\end{proof}

This completes the proof of the proposition too.
\end{proof}

The result below follows directly from Proposition~\ref{prop-reg}, Theorem~\ref{embed}, Proposition~\ref{clo-bound} and Theorem~\ref{afin-ultranormal}.

\begin{theorem}\label{ultranormal}
Every definable Hausdorff LC-space $X$, and a fortiori every definable manifold, is definably ultranormal.
\hspace*{\fill} $\Box$
\end{theorem}

A Hausdorff space $X$ is said to be {\em definably ultraparacompact} if every finite open definable cover $\{ U_1,\ldots,U_m \}$ can be
refined by a partition into a finitely many clopen definable sets; then, of course, there is a clopen definable cover $\{ \Omega_1,\ldots,\Omega_m \}$ partitioning the space $X$ such that $\Omega_i \subset U_i$ for all $i=1,\ldots,m$.

\vspace{1ex}

By an inductive argument (with respect to the cardinality $m$ of the open definable cover), Theorem~\ref{ultranormal} yields immediately the following

\begin{corollary}\label{ultrapara}
Every definable Hausdorff LC-space is definably ultraparacompact. In particular, so is every definable Hausdorff manifold.
\hspace*{\fill} $\Box$
\end{corollary}

\begin{remark}
In our paper~\cite{Now-Sym}, we gave a definable non-Archimedean version of Bierstone--Milman's desingularization algorithm from~\cite{BM}, which is a process of transforming an analytic function to normal crossings by blowing up along admissible smooth centers. It was done for a strong analytic function on a definably compact strong analytic manifold. The results of this section allow us to achieve the definable version of desingularization algorithm on arbitrary strong analytic manifolds. The proof can be repeated almost verbatim.
Let us recall that strong analyticity, being a model-theoretic strengthening of the weak non-Archimedean concept of analyticity (treated in the classical case e.g., by Serre~\cite{Se}), works well within definable settings, and makes it possible to apply a model-theoretic compactness argument in the absence of the ordinary topological compactness.
\end{remark}


In the recent paper~\cite{Now-closed}, we prove that, in an arbitrary 1-h-minimal structure $K$, every closed definable subset $A$ of $K^{n}$ is the zero locus of a continuous definable function $f:K^{n} \to K$, and is a definable retract of $K^{n}$. The latter yields immediately a definable non-Archimedean analogue of the Tietze--Urysohn extension theorem. Also, in our paper~\cite{Now-Lip}, we establish a theorem on definable Lipschitz extension of maps definable in an arbitrary 1-h-minimal structure. This may be regarded as a definable non-Archimedean version of Kirszbraun's extension theorem.

\vspace{1ex}

To our best knowledge, the only definable, non-Archimedean version of Kirszbraun's theorem was achieved by Cluckers--Martin~\cite{C-M} in the $p$-adic, thus locally compact case; more precisely, for Lipschitz extension of maps which are semi-algebraic, subanalytic or definable in an analytic structure on a finite extension of the field $\mathbb{Q}_{p}$ of $p$-adic numbers. 
The easier case of Lipschitz extension of definable $p$-adic maps on the line $\mathbb{Q}_{p}$ was treated in \cite{Kui}.

\vspace{2ex}

{\bf Acknowledgements.}
The author wishes to express his gratitude to the referees for valuable comments which greatly improved the previous version of the article.

\vspace{2ex}

\begin{small}
Institute of Mathematics

Faculty of Mathematics and Computer Science

Jagiellonian University

ul.~Profesora S.\ \L{}ojasiewicza 6 

30-348 Krak\'{o}w, Poland

{\em E-mail address: nowak@im.uj.edu.pl}
\end{small}

\end{document}